\documentclass[a4paper,11pt]{article}
\usepackage[english]{babel}
\usepackage{amsmath,amsfonts,amssymb,amsthm}
\usepackage[hyphens]{url}
\usepackage{tensor}
\usepackage{csquotes}
\usepackage{a4wide}

\newtheorem{theorem}{Theorem}

\newtheorem{corollary}{Corollary}

\theoremstyle{plain}
\newtheorem*{deffinition}{Definition}
\newtheorem*{propposition}{Proposition}
\newtheorem*{lemmma}{Lemma}
\newtheorem*{corrollary}{Corollary}
\theoremstyle{remark}

\newtheorem*{remmark}{Remark}

\begin{document}\sloppy

\title{On the endomorphism monoids of some groups with abelian automorphism group}

\author{Alexander Bors\thanks{The author is supported by the Austrian Science Fund (FWF):
Project F5504-N26, which is a part of the Special Research Program \enquote{Quasi-Monte Carlo Methods: Theory and Applications}. \newline 2010 \emph{Mathematics Subject Classification}: 20D15, 20D45, 20M32. \newline \emph{Key words and phrases:} groups with abelian automorphism group, endomorphism monoid, two-sided semidirect product}}

\date{}

\maketitle

\begin{abstract}
We investigate the endomorphism monoids of certain finite $p$-groups of order $p^8$ first studied by Jonah and Konvisser in 1975 as examples for finite $p$-groups with abelian automorphism group, and we show some necessary conditions for a finite $p$-group to have commutative endomorphism monoid. As a by-product, apart from formulas for the number of conjugacy classes of endomorphisms of said groups, we will be able to derive the following: There exist nonabelian groups with commutative endomorphism monoid, and having commutative endomorphism monoid is a group property strictly stronger than having abelian automorphism group. Furthermore, using a result of Curran, this will enable us to give, for all primes $p$, examples of finite $p$-groups which are direct products and have abelian automorphism group.
\end{abstract}

\section{Introduction}

For linguistical simplicity, let us follow the terminology used on the Groupprops Wiki \cite{GP} and call groups with abelian automorphism group \textit{aut-abelian} (and the corresponding abstract property \textit{aut-abelianity}). In analogy, in this paper we will call groups with commutative endomorphism monoid \textit{end-commutative} (and the property itself \textit{end-commutativity}). Consequences of aut-abelianity (and stronger assumptions, such as having an elementary abelian automorphism group) have been studied since the beginning of the 20th century; for a short history, see \cite{JRY13a}. There is still research going on in this field; for example, the paper \cite{JY12a} from 2012 gave, for the first time, examples of nonabelian finite $p$-groups with abelian automorphism group that are \textit{not} special.

\section{Some Technical Preliminaries}

We shall be using the following well-known facts: An aut-abelian group $G$ is nilpotent of class $2$ (in particular, since any finite nilpotent group is the direct product of its Sylow $p$-subgroups, the study of finite aut-abelian groups reduces to the one of finite aut-abelian $p$-groups) and all its automorphisms are central, meaning they are each of the form $x\mapsto x\cdot f(x)$ for some homomorphism $f:G\rightarrow\zeta G$. Conversely, for any group $G$ and any homomorphism $f:G\rightarrow\zeta G$, the map $G\rightarrow G,x\mapsto x\cdot f(x)$ (which we will denote by $f^{\ast}$), is an endomorphism of $G$, and it has trivial kernel if and only if $1$ is the only element in $\zeta G$ inverted by $f$.

From these observations, it is not difficult to see that a co-hopfian group $G$ is aut-abelian if and only if all its automorphisms are central and the homomorphisms $G\rightarrow\zeta G$ which only invert the identity pairwise commute in the endomorphism monoid. In particular, we can derive the following:

\begin{propposition}\label{endCommProp}
Let $G$ be a co-hopfian aut-abelian group such that all endomorphisms $\varphi$ of $G$ that are not automorphisms have their image contained in $\zeta G$ and only invert the identity. Then $G$ is end-commutative.
\end{propposition}

\begin{proof}
Let $\varphi,\psi$ be endomorphisms of $G$. We show that they commute in a case distinction. If both are automorphisms or both are non-automorphisms, this is clear by assumption and the observations before respectively. So assume w.l.o.g.~that $\varphi$ is an automorphism and $\psi$ is a non-automorphism of $G$. Let $\varphi=f^{\ast}$; then $f$ and $\psi$ commute, so for any $x\in G$: $$(\varphi\circ\psi)(x)=f^{\ast}(\psi(x))=\psi(x)\cdot f(\psi(x))=\psi(x)\cdot \psi(f(x))=\psi(f^{\ast}(x))=(\psi\circ\varphi)(x).$$
\end{proof}

This will be the basis for showing that the endomorphism monoids of certain groups introduced in the next section are commutative. In order to describe their structure, we need the following concept first introduced by Rhodes and Tilson in \cite{RT89a}. Let $M_1$ and $M_2$ be monoids, $M_1$ written additively with identity $0$ and $M_2$ multiplicatively with identity $1$. Furthermore, assume $M_2$ acts via monoid endomorphisms on $M_1$ both from the left and the right, that is, we have monoid homomorphisms $\varphi:M_2\rightarrow\mathrm{End}(M_1)$ and $\psi:M_2\rightarrow\mathrm{End}(M_1)^{\ast}$, where the asterisk denotes the opposite monoid. As usual, for $m_1\in M_1,m_2\in M_2$, we write $m_2m_1$ instead of $\varphi(m_2)(m_1)$ and $m_1m_2$ for $\psi(m_2)(m_1)$. Finally, assume that the actions given by $\varphi$ and $\psi$ commute, i.e., for all $m_1\in M_1$ and all $m_2,m_2'\in M_2$, we have $(m_2m_1)m_2'=m_2(m_1m_2')$. Then the \textit{two-sided semidirect product of $M_1$ and $M_2$ w.r.t.~$\varphi$ and $\psi$}, written $M_1\tensor[_\varphi]{\bowtie}{_\psi}M_2$, is defined to be the monoid with underlying set $M_1\times M_2$, the identity $(0,1)$ and the binary operation $(m_1,m_2)\odot(m_1',m_2'):=(m_1m_2'+m_2m_1',m_2m_2')$.

\section{Application to JK-groups}

We will now apply the proposition from the last section to certain groups studied by Jonah and Konvisser in 1975 in \cite{JK75a}. In their honor, let us make the following definition:

\begin{deffinition}\label{jkDef}
Let $p$ be a prime, and let $\lambda=(\lambda_1,\lambda_2)\in\{(1,0),(0,1),(1,1),\ldots,(p-1,1)\}$. The \textbf{JK-group} $G_{\lambda}^{(p)}$ is defined to be the $p$-group of class $2$ generated by four elements $a_1,a_2,b_1,b_2$ with the additional relations $$a_1^p=[a_1,b_1],a_2^p=[a_1,b_1^{\lambda_1}b_2^{\lambda_2}],b_1^p=[a_2,b_1b_2],b_2^p=[a_2,b_2],[a_1,a_2]=[b_1,b_2]=1.$$
\end{deffinition}

Using the standard van der Waerden idea of letting a group act on the set of supposed normal forms, it is not difficult to show that each element of $G_{\lambda}^{(p)}$ has a unique normal form representation as \begin{equation}\label{normalForm}a_1^{k_1}a_2^{k_2}b_1^{l_1}b_2^{l_2}[a_1,b_1]^{r_1}[a_1,b_2]^{r_2}[a_2,b_1]^{r_3}[a_2,b_2]^{r_4}\end{equation} with all exponents from the set $\{0,\ldots,p-1\}$. From this, one can derive that these are all special $p$-groups, with $\zeta G_{\lambda}^{(p)}=(G_{\lambda}^{(p)})'$ elementary abelian of order $p^4$ (with, for instance, $[a_1,b_1],[a_1,b_2],[a_2,b_1],[a_2,b_2]$ as an $\mathbb{F}_p$-basis), and the central quotient also elementary abelian of order $p^4$, with the images of $a_1,a_2,b_1,b_2$ under the canonical projection as an $\mathbb{F}_p$-basis. In view of this, any endomorphism of $G_{\lambda}^{(p)}$ whose image is contained in the center also has the center contained in its kernel (in particular only inverts the identity), whence the composition of any two such homomorphisms is the trivial endomorphism of $G_{\lambda}^{(p)}$. Thus for showing aut-abelianity of $G_{\lambda}^{(p)}$, Jonah and Konvisser \enquote{only} had to show that all automorphisms of $G_{\lambda}^{(p)}$ are central; their argument also gave that these are $p+1$ pairwise non-isomorphic groups of order $p^8$.

Still, all these considerations leave open the question as to whether an endomorphism of $G_{\lambda}^{(p)}$ whose image is \textit{not} contained in the center exists. In view of the proposition from Section 2, a negative answer to this question would immediately yield that $G_{\lambda}^{(p)}$ even is end-commutative. And indeed, this is the strategy for showing:

\begin{theorem}\label{jkTheo1}
The JK-groups $G_{\lambda}^{(p)}$ for $p\geq 3$ and $\lambda\not=(1,0)$ as well as $G_{(0,1)}^{(2)}$ all are end-commutative. In particular, there exist nonabelian finite $p$-groups with commutative endomorphism monoid.
\end{theorem}

On the other hand, it is not difficult to show:

\begin{theorem}\label{jkTheo2}
The JK-groups $G_{(1,0)}^{(p)}$ for all $p$ as well as $G_{(1,1)}^{(2)}$ are not end-commutative. In particular, end-commutativity is a strictly stronger property than aut-abelianity.
\end{theorem}

\begin{proof}[Proof of Theorem \ref{jkTheo2}.]
The counter-examples for commutativity of the endomorphism monoid for this proof were found using GAP \cite{GAP4}. For the $G_{(1,0)}^{(p)}$, consider the endomorphism $\varphi$ which sends each of the generators $a_1,a_2,b_1,b_2$ to the element $a_1a_2^{-1}$; it is not difficult to see by the defining relations that this really defines an endomorphism of $G_{(1,0)}^{(p)}$. Also, consider the automorphism $f^{\ast}$, where $f$ sends $a_1$ to $[a_1,a_2]$ and all other generators to $1$. Then $$(\varphi\circ f^{\ast})(a_1)=\varphi(a_1\cdot [a_1,a_2])=a_1a_2^{-1}\cdot [a_1a_2^{-1},a_1a_2^{-1}]=a_1a_2^{-1},$$ whereas $$(f^{\ast}\circ\varphi)(a_1)=\varphi(a_1)\cdot f(\varphi(a_1))=a_1a_2^{-1}\cdot f(a_1)f(a_2)^{-1}=a_1a_2^{-1}\cdot [a_1,a_2].$$

For $G_{(1,1)}^{(2)}$, consider the endomorphism $\varphi$ sending each generator to $a_1a_2b_2$; again, it is easy to check that this really defines an endomorphism. Also, consider, again, the automorphism $f^{\ast}$ where $f$ sends $a_1$ to $[a_1,a_2]$ and the three other generators to $1$. Then $$(\varphi\circ f^{\ast})(a_1)=\varphi(a_1[a_1,a_2])=a_1a_2b_2,$$ whereas $$(f^{\ast}\circ\varphi)(a_1)=\varphi(a_1)\cdot f(\varphi(a_1))=a_1a_2b_2\cdot [a_1,a_2].$$
\end{proof}

\begin{remmark}\label{firstRem}
Of course, Theorem \ref{jkTheo2} and its proof reveal almost nothing about the structure of the endomorphism monoid in those exceptional cases. We postpone an analysis to the final section.
\end{remmark}

Let us now prove the more difficult Theorem \ref{jkTheo1}. We begin with a lemma:

\begin{lemmma}\label{jkLem}
Let $p$ be an odd prime, $\lambda\in\{(0,1),(1,1),\ldots,(p-1,1)\}$. Then all elements of $G_{\lambda}^{(p)}$ that are outside the center have order $p^2$.
\end{lemmma}

\begin{proof}
Note that an element in normal form (\ref{normalForm}) has order $p$ if and only if the element $a_1^{k_1}a_2^{k_2}b_1^{l_1}b_2^{l_2}$ has order $p$. Also, by an easy induction on $n$, in all JK-groups and for all $k_1,k_2,l_1,l_2\in\{0,\ldots,p-1\}$, the $n$-th power $(a_1^{k_1}a_2^{k_2}b_1^{l_1}b_2^{l_2})^n$ has normal form $$a_1^{nk_1}a_2^{nk_2}b_1^{nl_1}b_2^{nl_2}[a_1,b_1]^{-\Delta_{n-1}k_1l_1}[a_1,b_2]^{-\Delta_{n-1}k_1l_2}[a_2,b_1]^{-\Delta_{n-1}k_2l_1}[a_2,b_2]^{-\Delta_{n-1}k_2l_2},$$ where $\Delta_k$ denotes the $k$-th triangle number and, of course, all exponents are understood to be reduced modulo $p$. Since $p$ is odd, it follows from this that $$(a_1^{k_1}a_2^{k_2}b_1^{l_1}b_2^{l_2})^p=[a_1,b_1]^{k_1+\lambda_1\cdot k_2}[a_1,b_2]^{\lambda_2\cdot k_2}[a_2,b_1]^{l_1}[a_2,b_2]^{l_1+l_2}.$$ Since $\lambda_2\not=0$, assuming that all exponents on the RHS are equal to $0$, we can conclude $k_1=k_2=l_1=l_2=0$.
\end{proof}

\begin{proof}[Proof of Theorem \ref{jkTheo1}]
For $G_{(0,1)}^{(2)}$, we checked this (or more precisely, whether all endomorphisms which are not automorphisms map into the center) with the help of GAP. The following concept helps keep the computational effort low: If $G$ is any group and $\varphi,\psi$ are endomorphisms of $G$ such that $\psi$ maps into $\zeta G$, then the sum $\varphi+\psi$ mapping $g\mapsto \varphi(g)\psi(g)$, is also an endomorphism of $G$. For endomorphisms of $G$, arising from each other by addition of an endomorphism mapping into the center of $G$ is an equivalence relation whose equivalence classes will be referred to as \textit{normalization classes}, and after fixing a set of representatives for the normalization classes, its elements will be referred to as the \textit{normalized endomorphisms} of $G$. Note that mapping into the center is a property of endomorphisms well-defined on normalization classes. Now in a JK-group, since the center and central quotient both are elementary abelian of order $p^4$, any function from $\{a_1,a_2,b_1,b_2\}$ into its center extends to an endomorphism of the group. Hence a set of normalized endomorphisms can be obtained by extending all functions from $\{a_1,a_2,b_1,b_2\}$ into any set of coset representatives of the center of $G_{(0,1)}^{(2)}$ (like the set of elements whose normal forms satisfy $r_i=0$ for $i=1,2,3,4$) to endomorphisms. Actually, since it is easy to see by the defining relations that any endomorphism of $G_{(0,1)}^{(2)}$ mapping at least one of the generators $a_1,a_2,b_1,b_2$ into the center maps all of them into it, one only needs to check for the $15^4=50625$ maps $\{a_1,a_2,b_1,b_2\}\rightarrow\{a_1^{k_1}a_2^{k_2}b_1^{l_1}b_2^{l_2}\mid(k_1,k_2,l_1,l_2)\in\{0,1\}^4\setminus\{(0,0,0,0)\}\}$ whether the map sending each generator to itself is the only one among them extending to an endomorphism of $G$.

For odd primes $p$, the following argument works: The (linear) map $$(\mathbb{Z}/p\mathbb{Z})^4\rightarrow(\mathbb{Z}/p\mathbb{Z})^4,(k_1,k_2,l_1,l_2)\mapsto (k_1+\lambda_1k_2,\lambda_2k_2,l_1,l_1+l_2),$$ is a bijection. In other words: Every element of $\zeta G_{\lambda}^{(p)}$ has exactly one $p$-th root of the form $a_1^{k_1}a_2^{k_2}b_1^{l_1}b_2^{l_2}$. Now if $\varphi$ is an endomorphism of $G_{\lambda}^{(p)}$ with nontrivial kernel, this implies that $\varphi$ sends at least one element of $G_{\lambda}^{(p)}\setminus \zeta G_{\lambda}^{(p)}$ into $\zeta G_{\lambda}^{(p)}$: Let $x\in\mathrm{ker}(\varphi)\setminus\{1\}$. If $x\notin \zeta G_{\lambda}^{(p)}$, $x$ is our example, and else, let $y\in G_{\lambda}^{(p)}\setminus \zeta G_{\lambda}^{(p)}$ with $y^p=x$; then, since $\varphi(y)$ has order at most $p$, $\varphi(y)\in\zeta G_{\lambda}^{(p)}$ by the lemma.

So $\varphi$ sends some element $g\in G_{\lambda}^{(p)}$ of the form $$a_1^{k_1}a_2^{k_2}b_1^{l_1}b_2^{l_2},\hspace{3pt}(k_1,k_2,l_1,l_2)\in\{0,\ldots,p-1\}^4\setminus\{(0,0,0,0)\},$$ into $\zeta G_{\lambda}^{(p)}$. Note that for showing $\mathrm{im}(\varphi)\subseteq\zeta G_{\lambda}^{(p)}$, by the defining relations it suffices to show that at least one of $a_1,a_2,b_1$ is mapped into the center. Also, we use the fact that for any element $h\in G_{\lambda}^{(p)}$, $\varphi(h)\in\zeta G_{\lambda}^{(p)}$ implies that all commutators of the form $[k,h]$ with $k\in G_{\lambda}^{(p)}$ are in $\mathrm{ker}(\varphi)$.

We now agree on the following notational convention: $$(k_1,k_2,l_1,l_2)\mapsto\zeta$$ is an abbreviation for $$\varphi(a_1^{k_1}a_2^{k_2}b_1^{l_1}b_2^{l_2})\in\zeta G_{\lambda}^{(p)},$$ and $$(K_1,K_2,L_1,L_2)\mapsto 1$$ for $$[a_1,b_1]^{K_1}[a_1,b_2]^{K_2}[a_2,b_1]^{L_1}[a_2,b_2]^{L_2}\in\mathrm{ker}(\varphi).$$ The observation from the end of the last paragraph then gives us the following implications: If $$(k_1,k_2,l_1,l_2)\mapsto\zeta,$$ we can conclude (by \enquote{taking brackets} with the generators) that $$(l_1,l_2,0,0)\mapsto 1,(0,0,l_1,l_2)\mapsto 1,(k_1,0,k_2,0)\mapsto 1\hspace{3pt}\text{and}\hspace{3pt}(0,k_1,0,k_2)\mapsto 1.$$ Also, the observation that an element is mapped into the center if and only if its $p$-th power is in the kernel of $\varphi$ translates as $$(k_1,k_2,l_1,l_2)\mapsto\zeta \Leftrightarrow (k_1+\lambda_1k_2,\lambda_2k_2,l_1,l_1+l_2)\mapsto 1,$$ or $$(K_1,K_2,L_1,L_2)\mapsto 1 \Leftrightarrow (K_1-\lambda_1\lambda_2^{-1}K_2,\lambda_2^{-1}K_2,L_1,L_2-L_1)\mapsto\zeta,$$ where $\lambda_2^{-1}$ denotes the unique multiplicative inverse modulo $p$ of $\lambda_2$ in $\{1,\ldots,p-1\}$.

We now make a case distinction according to the values of the exponents $k_1,k_2,l_1,l_2$ occurring in the normal form of $g$. First, assume that $l_1\not=0$. Since we have $$(0,0,l_1,l_2)\mapsto 1,$$ and this gives us $$(0,0,l_1,l_2-l_1)\mapsto\zeta,$$ which by taking brackets with $a_2$ yields $$(0,0,l_1,l_2-l_1)\mapsto 1,$$ we see that by iteration, we get $$(0,0,l_1,l)\mapsto 1$$ for all $l\in\{0,\ldots,p-1\}$, in particular $$(0,0,l_1,0)\mapsto 1,$$ which because of $l_1\not=0$ implies $$(0,0,1,0)\mapsto 1,$$ that is, $$(0,0,1,-1)\mapsto\zeta,$$ or explicitly, $$\varphi(b_1b_2^{-1})\in\zeta G_{\lambda}^{(p)}.$$ But also, by an appropriate subtraction among the $$(0,0,l_1,l)\mapsto 1,$$ we see that we can derive $$(0,0,0,1)\mapsto 1,$$ which yields $$(0,0,0,1)\mapsto\zeta,$$ or explicitly, $$\varphi(b_2)\in\zeta G_{\lambda}^{(p)}.$$ Combining the two results, we get that $$\varphi(b_1)\in\zeta G_{\lambda}^{(p)},$$ and are done in this case.

Now assume $l_1=0$. The subcase $k_2\not=0$ reduces to the first case, since $$(k_1,0,k_2,0)\mapsto 1$$ yields $$(k_1,0,k_2,-k_2)\mapsto\zeta.$$ Hence we can assume that $l_1=k_2=0$. But then the same argument yields $$(k_1,0,0,0)\mapsto\zeta,$$ which in case $k_1\not=0$ implies $$(1,0,0,0)\mapsto\zeta.$$ Hence we can assume $k_1=k_2=l_1=0$ and $l_2\not=0$. In this case, $$(l_1,l_2,0,0)\mapsto 1$$ simplifies to $$(0,l_2,0,0)\mapsto 1,$$ which yields $$(-\lambda_1\lambda_2^{-1}l_2,\lambda_2^{-1}l_2,0,0)\mapsto\zeta$$ and hence reduces the situation to the case $l_1=0,k_2\not=0$ dealt with before.
\end{proof}

\begin{remmark}\label{lastRem}
Note that we not only have shown in these cases that $\mathrm{End}(G_{\lambda}^{(p)})$ is commutative, but we also fully understand its structure: It consists of $2\cdot p^{16}$ elements, $p^{16}$ of which are automorphisms and $p^{16}$ homomorphisms into the center; the latter can be described in terms of the images of the four generators $a_1,a_2,b_1,b_2$ by $(4\times 4)$-matrices over $\mathbb{F}_p$. Also, it is not difficult to check for all JK-groups validity of the identity $f^{\ast}\circ g^{\ast}=(f+g)^{\ast}$ for all $f,g:G_{\lambda}^{(p)}\rightarrow \zeta G_{\lambda}^{(p)}$, from which it immediately follows that $\mathrm{Aut}(G_{\lambda}^{(p)})$ is elementary abelian of order $p^{16}$, as Jonah and Konvisser have already observed. On the other hand, for such $f,g$, we also have $f^{\ast}\circ g=g=g\circ f^{\ast}$ because of $f\circ g=0=g\circ f$. These observations show that in the cases considered in Theorem \ref{jkTheo1}, $\mathrm{End}(G_{\lambda}^{(p)})$ is isomorphic to $\mathrm{Mat}_{4,4}(\mathbb{F}_p)\tensor[_\varphi]{\bowtie}{_\varphi}\{0,1\}$, where $\varphi$ is the restriction of the usual action of the multiplicative monoid $\mathbb{F}_p$ on the additive monoid $\mathrm{Mat}_{4,4}(\mathbb{F}_p)$ (observe that this is both an action on the left and on the right because of commutativity).
\end{remmark}

\section{Some Necessary Conditions}

In this section, we show some results on how having commutative endomorphism monoid restricts the structure of finite $p$-groups. This will enable us to see \enquote{at first sight} for many examples with abelian automorphism group that their endomorphism monoid is not commutative.

\begin{propposition}\label{falsifyProp}
Let $p$ be a prime, $G$ a finite $p$-group, $e:=\mathrm{log}_p(\mathrm{exp}(G/G'))$.

(1) Assume $\mathrm{Aut}(G)\subseteq\zeta\mathrm{End}(G)$ (a property implied by end-commutativity). Then for all endomorphisms $\varphi$ of $G$, $\mathrm{im}(\varphi)$ is a characteristic subgroup of $G$. In particular, all cyclic subgroups of order a divisor of $p^e$ in $G$ are characteristic, and $\Omega_1(G)\leq\zeta G$.

(2) Assume $\mathrm{End}(G)$ is commutative. Then for all endomorphisms $\varphi$ of $G$, $\mathrm{im}(\varphi)$ is a fully-invariant subgroup of $G$. In particular, all cyclic subgroups of order a divisor of $p^e$ in $G$ are fully-invariant.

(3) If $\mathrm{Aut}(G)\subseteq\zeta\mathrm{End}(G)$, then $G$ cannot be written as a semidirect product in a nontrivial way. In particular, all endomorphisms $\varphi$ of $G$ are either automorphisms or nilpotent (meaning that some power $\varphi^n$ for $n$ a positive integer is the trivial endomorphism of $G$).
\end{propposition}

\begin{proof}
For (1): Let $\varphi$ be any endomorphism and $\alpha$ any automorphism of $G$. Take any element from $\mathrm{im}(\varphi)$, say $\varphi(g)$. Then $\alpha(\varphi(g))=\varphi(\alpha(g))\in\mathrm{im}(\varphi)$, which shows that $\mathrm{im}(\varphi)$ is characteristic. For the \enquote{In particular}, fix any cyclic subgroup $C$ of $G$ of order a divisor of $p^e$ and note that by the structure theorem for finite abelian groups, $G/G'$ is a product of primary cyclic groups of order a power of $p$, and the maximum of the orders of these factors is $p^e$. We consider the projection $\pi$ down onto one of these largest factors, and any homomorphism $\chi$ from that factor into $G$ sending the generator of the factor to a generator of $C$. Denoting by $\pi'$ the canonical projection $G\rightarrow G/G'$, the composition $\chi\circ\pi\circ\pi'$ then is an endomorphism of $G$ whose image is $C$, so that $C$ is characteristic. Finally, if $g$ is any element of order $p$ in $G$, then by what we just showed, $\langle g\rangle$ is characteristic, so if we fix any automorphism $\alpha$ of $G$, we will have $\alpha(g)=g^k$ for some $k\in\{1,\ldots,p-1\}$. However, by assumption, $\mathrm{Aut}(G)$ is abelian, so $\alpha$ is central, meaning that we must have $g^{k-1}\in\zeta G$. Hence if there is an automorphism of $G$ which does not fix $g$, then we can conclude $g\in\zeta G$, and otherwise, $g$ is in particular fixed by all inner automorphisms of $G$ and hence also an element of $\zeta G$.

For (2): This is analogous to (1).

For (3): Assume, for a contradiction, that $G=N\rtimes H$ for some nontrivial subgroups $N,H\leq G$. Then since $G/N\cong H$, there is an endomorphism $\varphi$ of $G$ with kernel $N$ and image $H$. Hence $H$ is characteristic in $G$, and $G=N\times H$. Now note that we have a sequence of groups $$G\twoheadrightarrow H\twoheadrightarrow H/H'\twoheadrightarrow \mathbb{Z}/p\mathbb{Z}\hookrightarrow \zeta N\hookrightarrow \zeta G,$$ and composition of all these morphisms yields a morphism $f:G\rightarrow\zeta G$. If $g=(n,h)\in G$ is inverted by $f$, because of $\mathrm{im}(f)\subseteq N$, it follows that $h=1$, but in view of the first morphism of the above chain, this implies $f(g)=1$, so $g=1$. Hence the associated central endomorphism $f^{\ast}$ is an automorphism of $G$, but it does not map $H$ into itself, a contradiction. For the \enquote{In particular}, observe that if $\varphi$ is an endomorphism of a finite group $G$, then setting $\mathrm{nil}(\varphi):=\{g\in G\mid \exists n\in\mathbb{N}:\varphi^n(g)=1_G\}$ and $\mathrm{per}(\varphi):=\{g\in G\mid \exists n\in\mathbb{N}^+:\varphi^n(g)=g\}$, we have $G=\mathrm{nil}(\varphi)\rtimes\mathrm{per}(\varphi)$ (see also the author's preprint \enquote{On the dynamics of endomorphisms of finite groups} on \url{http://arxiv.org/abs/1409.3756}), and $\varphi$ is nilpotent (resp.~an automorphism) if and only if $\mathrm{nil}(\varphi)=G$ (resp.~$\mathrm{per}(\varphi)=G$).
\end{proof}

\begin{remmark}\label{counterExRem}
The quaternion group $\mathrm{Q}_8$ has both the property that all images of its endomorphisms are fully-invariant subgroups and that it is not a semidirect product, but it does not even have abelian automorphism group. Hence not even the conjunction of these two conditions is sufficient to ensure aut-abelianity let alone end-commutativity.
\end{remmark}

\begin{corrollary}\label{curranCor}
Let $p$ be any prime, and let $\lambda\in\{(0,1),(1,1),\ldots,(p-1,1)\}$ with $\lambda\not=(1,1)$ for $p=2$. Furthermore, let $G$ be any aut-abelian special finite $p$-group. Then the direct product $G_{\lambda}^{(p)}\times G$ is aut-abelian.
\end{corrollary}

\begin{proof}
By Theorem \ref{jkTheo1}, the JK-group $G_{\lambda}^{(p)}$ has commutative endomorphism monoid and hence is directly indecomposable by point (3) of the above proposition. But Curran in \cite{Cur06a} proved (this is point (i) of Theorem 3.1 there) that if $H$ and $K$ are special finite aut-abelian $p$-groups which have no common direct factor, then $H\times K$ is aut-abelian as well; applying this with $H:=G_{\lambda}^{(p)}$ and $K:=G$, we are done.
\end{proof}

\begin{remmark}\label{falsifyRem}
As mentioned earlier, one can use the above proposition to falsify end-commutativity quickly in many known examples of finite aut-abelian $p$-groups:

(1) The first example of a finite non-abelian but aut-abelian group was given by Miller in \cite{Mil13a} from 1913. It is a $2$-group of order $64$. Later, in \cite{Str82a} from 1982, Struik generalized this example by giving, for each integer $n\geq 3$, a non-abelian but aut-abelian group $G_n$ of order $2^{n+3}$ such that Miller's example is $G_3$. As Struik observed, these groups are semidirect products, and hence they are not end-commutative by point (3) of the proposition. The same argument applies, of course, to the examples constructed in \cite{Cur87a}.

(2) In \cite{Mor94a} from 1994, Morigi constructed, for each prime $p$ and positive integer $n$, a finite $p$-group $G^{(p)}(n)$ of order $p^{n^2+3n+3}$ whose automorphism group is elementary abelian of order $p^{(n^2+n+1)(2n+2)}$. Since she also showed in that paper that for odd $p$, there is no aut-abelian group of order $p^6$ yet setting $n=1$ gives an aut-abelian group of order $p^7$, this settled the question for the minimal order of finite non-abelian aut-abelian $p$-groups for odd $p$. $G^{(p)}(n)$ is a finite $p$-group of class $2$ generated by elements $a_1,a_2,b_1,\ldots,b_{2n}$, and from the defining relations as well as the derived normal form theorem, it follows that both $a_1$ and $a_2$ are elements of order $p$ which are not in the center, whence by (1) of the proposition, $G^{(p)}(n)$ is not end-commutative. The group $G$ from \cite{JY12a}, which was the first example of a non-special finite aut-abelian $p$-group, also has an element of order $p$ not contained in the center (namely, for instance, the generator $x_4$).

(3) Later, in 2013, further examples of finite non-special aut-abelian $p$-groups were constructed in \cite{JRY13a}. For the group $G$ defined there, it is not difficult to see that $G/G'\cong\mathbb{Z}/p^{n-2}\mathbb{Z}\times(\mathbb{Z}/p^2\mathbb{Z})^3$ so that in particular $\mathrm{exp}(G/G')=p^{n-2}\geq p^2$ and there exists an automorphism mapping $x_4$ to, say, $x_4[x_1,x_2]$ so that the cyclic subgroup of order $p^2$ generated by $x_4$ is not characteristic. The group $G_1$ has an element of order $p$ not contained in the center, namely for instance $x_5$. Finally, as for the group $G_2$, note that the image of the generator $x_3$ under the canonical projection $G_2\rightarrow G_2/G_2'$ still has order $p^2$, so the exponent of the abelianization is $p^2$. But clearly, there exists an automorphism mapping $x_3$ to $x_3[x_1,x_3]$, which is not contained in the cyclic subgroup generated by $x_3$.
\end{remmark}

\section{The structure of $\mathrm{End}(G_{\lambda}^{(p)})$ in the exceptional cases}

For completeness, let us also study the endomorphism monoids in the remaining cases (the ones considered in Theorem \ref{jkTheo2}). We begin with the case $\lambda=(1,0)$ and $p$ an arbitrary prime; the case $\lambda=(1,1),p=2$ behaves differently. Throughout this whole section, the notion of a normalized endomorphism of $G_{\lambda}^{(p)}$ is understood as an endomorphism mapping each of the generators $a_1,a_2,b_1,b_2$ to an element in whose normal form the exponents of the commutators are all $0$; we denote this set of coset representatives of the center by $\mathcal{R}_{\lambda}^{(p)}$. The essential observation is the following:

\begin{theorem}\label{exceptionalTheo}
Let $p$ be any prime and let $a_1,a_2,b_1,b_2$ denote the generators of $G_{(1,0)}^{(p)}$ from its presentation. The following are equivalent for all functions $f:\{a_1,a_2,b_1,b_2\}\rightarrow\mathcal{R}_{(1,0)}^{(p)}$:

(1) $f$ extends to a (normalized) endomorphism of $G_{(1,0)}^{(p)}$.

(2) $f$ either maps each generator to itself or its image is contained in the cyclic subgroup $\langle a_1a_2^{-1}\rangle$.
\end{theorem}

Note that for each $p$, since the implication \enquote{(2)$\Rightarrow$(1)} in Theorem \ref{exceptionalTheo} is easy to check, the theorem holds true if and only if the number of normalized endomorphisms of $G_{(1,0)}^{(p)}$ equals $p^4+1$, which allowed us to check the case $p=2$ with GAP, so we may assume $p\geq 3$ from now on. In this case, using the observations on normal forms of powers from the proof of the lemma in Section 2, it is easy to see that elements in normal form have the same $p$-th power if and only if the sum of the exponents of $a_1$ and $a_2$ is the same modulo $p$ in both cases. In particular, the set of all elements of $G_{(1,0)}^{(p)}$ whose order divides $p$, which already coincides with $\Omega_1(G)$, is the subgroup generated by the center and the one element $a_1a_2^{-1}$, so an equivalent way of stating the (hard) implication \enquote{(1)$\Rightarrow$(2)} in Theorem \ref{exceptionalTheo} is that all endomorphisms of $G_{(1,0)}^{(p)}$ that are not automorphisms map into $\Omega_1(G)$ (which is an abelian subgroup). This is in principle just as for the other $G_{\lambda}^{(p)}$, but the set of potential images is a bit more complicated and also, the argumentation from the proof of Theorem \ref{jkTheo1} does not work here.

In order to prove Theorem \ref{exceptionalTheo}, we will use the following lemma (which is also part of Jonah and Konvisser's observations to show aut-abelianity):

\begin{lemmma}\label{commLem}
Let $p$ be an odd prime, $\lambda\in\{(1,0),(0,1),(1,1),\ldots,(p-1,1)\}$. The following are equivalent for all tuples $(k_1,k_2,l_1,l_2),(K_1,K_2,L_1,L_2)\in(\mathbb{Z}/p\mathbb{Z})^4$:

(1) The elements $a_1^{k_1}a_2^{k_2}b_1^{l_1}b_2^{l_2},a_1^{K_1}a_2^{K_2}b_1^{L_1}b_2^{L_2}\in G_{\lambda}^{(p)}$ commute.

(2) $(k_1,k_2,l_1,l_2)=0$ or $(K_1,K_2,L_1,L_2)=0$ or $l_1=l_2=L_1=L_2=0$ or $k_1=k_2=K_1=K_2=0$ or there exists $n\in\mathbb{Z}/p\mathbb{Z}$ such that $(K_1,K_2,L_1,L_2)=n\cdot(k_1,k_2,l_1,l_2)$.
\end{lemmma}

\begin{proof}
Considering the commutator of the two elements yields the following: The elements in question commute if and only if the following four relations hold true: $$k_1L_1=l_1K_1,k_1L_2=l_2K_1,k_2L_1=l_1K_2,k_2L_2=l_2K_2.$$ From this, verifying \enquote{(2)$\Rightarrow$(1)} consists only in some routine checks, and for \enquote{(1)$\Rightarrow$(2)}, assume that of the five cases mentioned in (2), the first four are violated; assume w.l.o.g.~that $k_1\not=0$. Set $n:=\frac{K_1}{k_1}$ in $\mathbb{F}_p$. Then $K_1=n\cdot k_1$ is clear by definition, and it follows directly from the above relations that $L_1=n\cdot l_1$ and $L_2=n\cdot l_2$. Since at least one of $l_1,l_2,L_1,L_2$ is not zero, we therefore cannot have $l_1=l_2=0$, so assume w.l.o.g.~that $l_1\not=0$. Then the third relation gives $K_2=\frac{L_1}{l_1}k_2=n\cdot k_2$.
\end{proof}

\begin{proof}[Proof of Theorem \ref{exceptionalTheo}.]
As mentioned before, the check of \enquote{(2)$\Rightarrow$(1)} is easy, so in this proof, we concentrate on the implication \enquote{(1)$\Rightarrow$(2)} for odd $p$. Note that it is not difficult to see that for an endomorphism $\varphi$ of $G_{(1,0)}^{(p)}$, the following are equivalent:

(1) $\varphi$ maps into $\Omega_1(G_{(1,0)}^{(p)})$.

(2) $(G_{(1,0)}^{(p)})'=\zeta G_{(1,0)}^{(p)}\subseteq\mathrm{ker}(\varphi)$.

(3) $\mathrm{im}(\varphi)$ is abelian.

Also, for normalized $\varphi$, (1) can be restated as

(1') $\varphi$ maps the generators $a_1,a_2,b_1,b_2$ into $\langle a_1a_2^{-1}\rangle$.

In what follows, just as Jonah and Konvisser, we denote by $A$ the subgroup of $G_{(1,0)}^{(p)}$ generated by the center and the two generators $a_1,a_2$ and by $B$ the subgroup generated by the center and $b_1,b_2$; observe that $A$ and $B$ are abelian subgroups.

Now fix a normalized endomorphism $\varphi$ of $G_{(1,0)}^{(p)}$. We make use of the fact that homomorphisms preserve relations; more specifically, since the relations $[a_1,a_2]=1$ and $a_1^p=a_2^p$ hold true in $G_{(1,0)}^{(p)}$, the images $\varphi(a_1),\varphi(a_2)$ must also commute and have the same $p$-th powers. The latter consequence implies that the two images must arise from each other by left multiplication with an element of the cyclic subgroup $\langle a_1a_2^{-1}\rangle$, which in view of the first consequence and the lemma proved before shows that precisely one of the following two cases occurs:

(1.) $\varphi(a_1),\varphi(a_2)\in A$ and are not necessarily equal, but can only differ by (left) multiplication with a power of $a_1a_2^{-1}$.

(2.) $\varphi(a_1)=\varphi(a_2)\notin A$.

We now consider these cases one after the other, starting with (1.). Assume that $\varphi$ does not map into $\langle a_1a_2^{-1}\rangle$. We need to show that $\varphi$ fixes each generator. Say $$a_1\mapsto a_1^{k_1}a_2^{k_2},a_2\mapsto a_1^ka_2^{k_1+k_2-k},b_1\mapsto a_1^{K_1}a_2^{K_2}b_1^{L_1}b_2^{L_2},b_2\mapsto a_1^{\kappa_1}a_2^{\kappa_2}b_1^{\lambda_1}b_2^{\lambda_2}$$ for appropriate $k_1,k_2,k,K_1,K_2,L_1,L_2,\kappa_1,\kappa_2,\lambda_1,\lambda_2\in\mathbb{Z}/p\mathbb{Z}$. Note that $\varphi(a_2)\not=1$, since otherwise by the defining relations, $b_1,b_2$ map to elements of order a divisor of $p$, which are all in the centralizer of $A$ so that the image of $\varphi$ then is abelian, a contradiction.

Let us now show that at least one of $k_1,k_2$ is not equal to $0$. Otherwise, $a_1\mapsto 1$ and by the observation from the end of the last paragraph, $k\not=0$. Note that by the fourth of the defining relations, we have $$[a_1,b_1]^{\kappa_1+\kappa_2}[a_2,b_1]^{\lambda_1}[a_2,b_2]^{\lambda_1+\lambda_2}=[a_1,b_1]^{\lambda_1\cdot k}[a_1,b_2]^{\lambda_2\cdot k}[a_2,b_1]^{\lambda_1\cdot k}[a_2,b_2]^{\lambda_2\cdot k}.$$ Comparing the exponents, we obtain that $\lambda_2=0$ and $\lambda_1=0$ so that $b_2$ maps into $A$. Since $\varphi(b_1)$ and $\varphi(b_2)$ must commute, if $\varphi(b_2)\not=1$, $\varphi(b_1)\in A$ as well so that the image is abelian. But if $\varphi(b_2)=1$, the third defining relation yields the same as the fourth with $\kappa_i$ replaced by $K_i$ and $\lambda_j$ by $L_j$, so that we then get $L_1=L_2=0$ and the image is abelian in this case as well. Hence we indeed get that $(k_1,k_2)\not=(0,0)$.

The first defining relation yields $$[a_1,b_1]^{k_1+k_2}=[a_1,b_1]^{k_1L_1}[a_1,b_2]^{k_1L_2}[a_2,b_1]^{k_2L_1}[a_2,b_2]^{k_2L_2},$$ hence by comparison of exponents $$k_1L_1=k_1+k_2,k_1L_2=k_2L_1=k_2L_2=0.$$ If $k_2\not=0$, it follows from this that $L_1=L_2=0$ so that $\varphi(b_1)\in A$ and the image certainly is abelian except possibly for the case $\varphi(b_1)=1$. However, in this case, by the second defining relation, $a_2$ must map to some element of order $p$, that is, $\varphi(a_2)=a_1^na_2^{-n}$ for some $n\not=0$, and then, just as before, by considering the fourth defining relation, we see that actually, $\varphi(b_2)\in A$ as well.

Hence $k_2=0$ so that $k_1\not=0$, which by the first defining relation implies $L_1=1$ and $L_2=0$. Summing up, so far we know that $$a_1\mapsto a_1^{k_1},a_2\mapsto a_1^ka_2^{k_1-k},b_1\mapsto a_1^{K_1}a_2^{K_2}b_1,b_2\mapsto a_1^{\kappa_1}a_2^{\kappa_2}b_1^{\lambda_1}b_2^{\lambda_2}$$ for appropriate $k_1,k,K_1,K_2,\kappa_1,\kappa_2,\lambda_1,\lambda_2\in\mathbb{Z}/p\mathbb{Z}$ with $k_1\not=0$.

Now the fourth defining relation yields $$[a_1,b_1]^{\kappa_1+\kappa_2}[a_2,b_1]^{\lambda_1}[a_2,b_2]^{\lambda_1+\lambda_2}=[a_1,b_1]^{k\cdot\lambda_1}[a_1,b_2]^{k\cdot\lambda_2}[a_2,b_1]^{(k_1-k)\cdot\lambda_1}[a_2,b_2]^{(k_1-k)\cdot\lambda_2},$$ hence $$\kappa_1+\kappa_2=k\cdot\lambda_1,k\cdot\lambda_2=0,(k_1-k)\cdot\lambda_1=\lambda_1,(k_1-k)\cdot\lambda_2=\lambda_1+\lambda_2.$$ If $k\not=0$, it follows from this that $\lambda_2=0$ and $\lambda_1=0$ so that then $\varphi(b_2)\in A$ and we are done except possibly for the case $\varphi(b_2)=1$. However, in this case, a look at the third defining relation yields $$[a_1,b_1]^{K_1+K_2}[a_2,b_1][a_2,b_2]=[a_1,b_1]^k[a_2,b_1]^{k_1-k},$$ a contradiction.

Hence $k=0$, which lets us simplify the above identities to $$\kappa_2=-\kappa_1,(k_1-1)\cdot\lambda_1=0,(k_1-1)\cdot\lambda_2=\lambda_1.$$ From this, it follows that if $k_1\not=1$, then $\lambda_1=\lambda_2=0$ so that $\varphi(b_2)\in A$, and $\varphi(b_1)\in A$ follows as in the last paragraph. Therefore $k_1=1$, from which it also follows that $\lambda_1=0$.

By now, we know that $$a_1\mapsto a_1,a_2\mapsto a_2,b_1\mapsto a_1^{K_1}a_2^{K_2}b_1,b_2\mapsto a_1^{\kappa_1}a_2^{-\kappa_1}b_2^{\lambda_2}$$ for appropriate $K_1,K_2,\kappa_1,\lambda_2\in\mathbb{Z}/p\mathbb{Z}$. The third defining relation yields $$[a_1,b_1]^{K_1+K_2}[a_2,b_1][a_2,b_2]=[a_2,b_1][a_2,b_2]^{\lambda_2},$$ that is, $K_2=-K_1$ and $\lambda_2=1$. Now by the form of $\varphi(b_1)=a_1^{K_1}a_2^{-K_1}b_1$ and $\varphi(b_2)=a_1^{\kappa_1}a_2^{-\kappa_1}b_2$ and the fact that they must commute, we can conclude by the lemma before that $K_1=\kappa_1=0$ so that $\varphi$ indeed fixes all of the generators. This concludes the proof of \enquote{(1)$\Rightarrow$(2)} in case (1.).

It remains to treat the case (2.), where neither of the images $\varphi(a_1),\varphi(a_2)$ is contained in $A$, but we know that they are equal. In this case, we want to show that $\varphi$ maps the generators into $\langle a_1a_2^{-1}\rangle$, or equivalently, that the center is contained in the kernel of $\varphi$. Say $$a_1,a_2\mapsto a_1^{k_1}a_2^{k_2}b_1^{l_1}b_2^{l_2},b_1\mapsto a_1^{K_1}a_2^{K_2}b_1^{L_1}b_2^{L_2},b_2\mapsto a_1^{\kappa_1}a_2^{\kappa_2}b_1^{\lambda_1}b_2^{\lambda_2}$$ for appropriate $k_1,k_2,l_1,l_2,K_1,K_2,L_1,L_2,\kappa_1,\kappa_2,\lambda_1,\lambda_2\in\mathbb{Z}/p\mathbb{Z}$ such that $(l_1,l_2)\not=(0,0)$. Note that also, $\varphi(b_1)\not=1$, since otherwise, by the defining relations, $\varphi(a_1),\varphi(a_2)$ are forced to have order $p$, from which it follows that they are contained in $A$ (since $\varphi$ is normalized), a contradiction.

Now we know that $\varphi(b_1)$ and $\varphi(b_2)$ must commute. Note that since $\varphi(a_1)=\varphi(a_2)$, the defining relations give $\varphi(b_1)^p=\varphi(a_1)^p\varphi(b_2)^p$, that is, $$[a_1,b_1]^{K_1+K_2}[a_2,b_1]^{L_1}[a_2,b_2]^{L_1+L_2}=[a_1,b_1]^{k_1+k_2+\kappa_1+\kappa_2}[a_2,b_1]^{l_1+\lambda_1}[a_2,b_2]^{l_1+l_2+\lambda_1+\lambda_2},$$ that is, $$K_1+K_2=k_1+k_2+\kappa_1+\kappa_2,L_1=l_1+\lambda_1,L_2=l_2+\lambda_2.$$ Hence if $\varphi(b_1)\in A$, we get a contradiction, since by the lemma, we necessarily have $\varphi(b_2)\in A$ as well, which is impossible since at least one $l_i\not=0$. Therefore, we need to deal with the following two subcases:

(a.) $\varphi(b_1)\notin A\cup B$, and the tuple of exponents of $\varphi(b_2)$ is a multiple of the one of $\varphi(b_1)$.

(b.) $\varphi(b_1),\varphi(b_2)\in B$.

Let us start with (a.). Then we know that $\mathrm{ker}(\varphi)$ contains the elements $a_1a_2^{-1}$ and $b_1^kb_2^{-1}c$ for some $c\in\zeta G_{(1,0)}^{(p)}$ and $k\in\mathbb{Z}/p\mathbb{Z}$. Taking brackets with generators and raising the second element to the $p$-th power shows that $\mathrm{ker}(\varphi)$ contains the following elements from the center: $$[a_1,b_1][a_2,b_1]^{-1},[a_1,b_2][a_2,b_2]^{-1},[a_1,b_1]^k[a_1,b_2]^{-1},[a_2,b_1]^k[a_2,b_2]^{k-1}.$$ Since we want to show that $\mathrm{ker}(\varphi)$ contains the entire center, we want to know whether these four elements generate the center. Now computing the determinant of the associated $(4\times 4)$-matrix over $\mathbb{F}_p$ by any of the standard methods yields: $$\mathrm{det}\begin{pmatrix}1 & 0 & k & 0 \\ 0 & 1 & -1 & 0 \\ -1 & 0 & 0 & k \\ 0 & -1 & 0 & k-1\end{pmatrix}=k^2,$$ so except for the case $k=0$, we are done.

Now $k=0$ means that $\varphi(b_2)=1$. In this case, $\varphi(b_1)^p=\varphi(a_1)^p$, from which it follows that $K_2=k_1+k_2-K_1,L_1=l_1,L_2=l_2$. Then the first (or second) defining relation yields $$k_1+k_2=(k_1-K_1)l_1,0=(k_1-K_1)l_2,l_1=(k_2-K_2)l_1,l_1+l_2=(k_2-K_2)l_2.$$ In particular, we have $l_2=0$ or $K_1=k_1$. However, $l_2=0$ implies by the last of the identities that $l_1=0$, a contradiction. Hence we can conclude $K_1=k_1$ so that $a_1,a_2,b_1$ all map to the same value and the image of $\varphi$ is abelian, as we wanted to show. This concludes the proof of the implication in subcase (a.).

Subcase (b.) means that $K_1=K_2=\kappa_1=\kappa_2=0$. Then the first of the identities which we derived at the beginning of the argumentation for case (2.) (before the distinction of the two subcases) yields $k_2=-k_1$, and then the first defining relation gives $$0=k_1L_1,0=k_1L_2,l_1=-k_1L_1,l_1+l_2=-k_1L_2.$$ Now if $k_1\not=0$, then the first two identities give $L_1=L_2=0$, which, plugged into the last two identities, yields $l_1=l_2=0$, a contradiction. However, $k_1=0$ also yields $l_1=l_2=0$, so the subcase (b.) is impossible. This concludes the proof of Theorem \ref{exceptionalTheo}.
\end{proof}

Now that we know all endomorphisms of $G_{(1,0)}^{(p)}$, we can also exhibit the structure of the monoid which they form. Consider the following binary operation $\odot$ on the set $(\mathbb{Z}/p\mathbb{Z})^4$: $$(x_1,x_2,x_3,x_4)\odot(y_1,y_2,y_3,y_4):=((x_1-x_2)y_1,(x_1-x_2)y_2,(x_1-x_2)y_3,(x_1-x_2)y_4).$$ It is not difficult to see (and will also follow from the proof of the next corollary) that $(\mathbb{Z}/p\mathbb{Z})^4$ together with this operation is a semigroup, which we will denote henceforth by $S_p$. Note that $S_p$ has no identity (more specifically, there exists no right-neutral element for elements distinct from $(0,0,0,0)$ whose first two entries are equal), and recall that for any semigroup $S$, $S^1$ denotes the monoid obtained from $S$ by adjoining an identity if $S$ itself does not have one (and otherwise $S^1:=S$).

Let $\phi$ denote the left action of the monoid $S_p^1$ on the additive monoid $\mathrm{Mat}_{4,4}(\mathbb{F}_p)$ where all non-identity elements act as the zero endomorphism, and let $\psi$ denote the right action of $S_p^1$ on $\mathrm{Mat}_{4,4}(\mathbb{F}_p)$ defined via $$(a_{i,j})_{1\leq i,j\leq 4}\cdot(x_1,x_2,x_3,x_4)^t:=(x_j(a_{i,1}-a_{i,2}))_{1\leq i,j\leq 4}.$$ Observe that every endomorphism $\varphi$ of $G_{(1,0)}^{(p)}$ has a unique representation as a sum of a normalized endomorphism and an endomorphism mapping into the center; we denote the first by $\mathrm{norm}(\varphi)$ and the latter by $\mathrm{cent}(\varphi)$. To each normalized endomorphism $\tilde{\varphi}$, which by Theorem \ref{exceptionalTheo} is either the identity or maps each generator into $\langle a_1a_2^{-1}\rangle$, we associate an element $\mathrm{vec}(\tilde{\varphi})$ from $S_p^1$, namely the identity if $\tilde{\varphi}=\mathrm{id}$ and otherwise the vector consisting of the exponents occurring when writing the values of $a_1,a_2,b_1,b_2$ under $\varphi$ in that order as powers of $a_1a_2^{-1}$. Also, to each endomorphism $f$ whose image is contained in the center, we associate a $(4\times 4)$-matrix over $\mathbb{F}_p$, denoted $\mathrm{mat}(f)$, whose $j$-th column is the coordinate vector of the image of the $j$-th generator (with the generators again in the \enquote{usual} order $a_1,a_2,b_1,b_2$) in terms of the basis $[a_1,b_1],[a_1,b_2],[a_2,b_1],[a_2,b_2]$. Then we have the following:

\begin{corollary}\label{exceptionalCor}
The map $\alpha:\mathrm{End}(G_{(1,0)}^{(p)})\rightarrow \mathrm{Mat}_{4,4}(\mathbb{F}_p)\tensor[_\phi]{\bowtie}{_\psi}S_p^1$ sending $\varphi\mapsto (\mathrm{mat}(\mathrm{cent}(\varphi)),$\linebreak[4]$\mathrm{vec}(\mathrm{norm}(\varphi))),$ is an isomorphism of monoids.
\end{corollary}

\begin{proof}
The map clearly is a bijection sending the identity of its source to the identity of the target, so it suffices to show compatibility with the binary operations. Now for any endomorphisms $\varphi_1,\varphi_2$ of $G_{(1,0)}^{(p)}$, we have $$\alpha(\varphi_1\circ\varphi_2)=\alpha((\mathrm{norm}(\varphi_1)+\mathrm{cent}(\varphi_1))\circ(\mathrm{norm}(\varphi_2)+\mathrm{cent}(\varphi_2)))=$$$$\alpha(\mathrm{norm}(\varphi_1)\circ\mathrm{norm}(\varphi_2)+\mathrm{cent}(\varphi_1)\circ\mathrm{norm}(\varphi_2)+\mathrm{norm}(\varphi_1)\circ\mathrm{cent}(\varphi_2))=$$$$(\mathrm{mat}(\mathrm{cent}(\varphi_1)\circ\mathrm{norm}(\varphi_2))+\mathrm{mat}(\mathrm{norm}(\varphi_1)\circ\mathrm{cent}(\varphi_2)),\mathrm{vec}(\mathrm{norm}(\varphi_1)\circ\mathrm{norm}(\varphi_2));$$

for the last equality, note that since the center is equal to the commutator subgroup, it is fully-invariant, so that $\mathrm{norm}(\varphi_1)\circ\mathrm{cent}(\varphi_2)$ is an endomorphism mapping into the center. Also, an endomorphism is normalized if and only if it maps all generators into $\langle a_1a_2^{-1}\rangle$, a property clearly preserved under composition.

Now let $\mathrm{vec}(\mathrm{norm}(\varphi_k))=(x_{k,1},x_{k,2},x_{k,3},x_{k,4})^t$ and $\mathrm{mat}(\mathrm{cent}(\varphi_k))=(a_{i,j}^{(k)})_{1\leq i,j\leq 4}$ for $k=1,2$. Then $\mathrm{norm}(\varphi_1)\circ\mathrm{norm}(\varphi_2)$ sends the $j$-th generator to $$\mathrm{norm}(\varphi_1)((a_1a_2^{-1})^{x_{2,j}})=((a_1a_2^{-1})^{x_{1,1}}(a_1a_2^{-1})^{-x_{1,2}})^{x_{2,j}}=(a_1a_2^{-1})^{(x_{1,1}-x_{1,2})x_{2,j}}.$$ Also, since endomorphisms which are not automorphisms have the entire center contained in their kernel, $\mathrm{norm}(\varphi_1)\circ\mathrm{cent}(\varphi_2)$ is the trivial endomorphism if $\mathrm{norm}(\varphi_1)\not=\mathrm{id}$, and $\mathrm{cent}(\varphi_2)$ otherwise. Finally, $\mathrm{cent}(\varphi_1)\circ\mathrm{norm}(\varphi_2)$ maps the $j$-th generator to $\mathrm{cent}(\varphi_1)((a_1a_2^{-1})^{x_{2,j}})$, which is $$[a_1,b_1]^{x_{2,j}(a_{1,1}^{(1)}-a_{1,2}^{(1)})}[a_1,b_2]^{x_{2,j}(a_{2,1}^{(1)}-a_{2,2}^{(1)})}[a_2,b_1]^{x_{2,j}(a_{3,1}^{(1)}-a_{3,2}^{(1)})}[a_2,b_2]^{x_{2,j}(a_{4,1}^{(1)}-a_{4,2}^{(1)})}.$$ From all this, the assertion now follows from the definitions of \enquote{two-sided semidirect product} and of the structures involved.
\end{proof}

\begin{corollary}\label{exceptionalCor2}
Let $p$ be any prime. The action of $\mathrm{Aut}(G_{(1,0)}^{(p)})=\mathrm{End}(G_{(1,0)}^{(p)})^{\ast}$ on $\mathrm{End}(G_{(1,0)}^{(p)})$ via conjugation has exactly $3p^{16}-p^{12}$ orbits, of which $2p^{16}$ have length $1$ and $p^{12}(p^4-1)$ have length $p^4$.
\end{corollary}

\begin{proof}
Note that since the automorphism group of $G_{(1,0)}^{(p)}$ is abelian, under this action it decomposes into $p^{16}$ orbits of length $1$, so we can concentrate on the conjugacy classes of endomorphisms which are not automorphisms. Since automorphisms of $G_{(1,0)}^{(p)}$ are precisely the endomorphisms of the form $\mathrm{id}+f$ for some endomorphism $f$ mapping into the center, and $(\mathrm{id}+f)^{-1}=\mathrm{id}-f$, we obtain that for any $\varphi=\tilde{\varphi}+g\in\mathrm{End}(G_{(1,0)}^{(p)})\setminus\mathrm{Aut}(G_{(1,0)}^{(p)})$, the conjugates of $\varphi$ are the endomorphisms of the form $$(\mathrm{id}+f)(\tilde{\varphi}+g)(\mathrm{id}-f)=(\mathrm{id}+f)(\tilde{\varphi}+g)=\tilde{\varphi}+g+f\tilde{\varphi}.$$ The possible variation comes from the summand $f\tilde{\varphi}$; if $\tilde{\varphi}=0$, then no variation is possible and we obtain another $p^{16}$ orbits of length $1$. However, if $\tilde{\varphi}\not=0$, then $\tilde{\varphi}$ is represented by a non-zero vector, and in view of the mechanics of the right action of normalized endomorphisms on endomorphisms mapping into the center, we obtain only orbits of length $p^4$ then, since for any $i\in\{1,\ldots,4\}$ such that the $i$-th entry of $\mathrm{vec}(\tilde{\varphi})$ is not zero, we can freely control the $i$-th column of $f\tilde{\varphi}$ by appropriate choices of $f$, but for each such choice, the other columns are determined.
\end{proof}

It remains to discuss the group $G_{(1,1)}^{(2)}$. It is not difficult to see that for this group, the set of elements of order at most $2$ once again coincides with the first omega subgroup and is equal to $\langle\zeta G_{(1,1)}^{(2)},a_1a_2b_2\rangle$, and in analogy to the other exceptional cases, we have, in addition to the automorphism group, which is the normalization class of $\mathrm{id}$, $2^4=16$ normalization classes whose normalized representatives map all generators into $\langle a_1a_2b_2\rangle$. However, a search with GAP yielded that $G_{(1,1)}^{(2)}$ has 23 normalization classes of endomorphisms altogether. The six additional normalized endomorphisms are as follows:

\begin{center}
\begin{tabular}{|c|c|c|c|c|}
\hline
endomorphism & image of $a_1$ & image of $a_2$ & image of $b_1$ & image of $b_2$ \\ \hline
$\varphi_1$ & $a_2b_1b_2$ & $a_2b_1b_2$ & $a_1b_2$ & $1$ \\ \hline
$\varphi_2$ & $a_2b_1b_2$ & $a_2b_1b_2$ & $a_1a_2b_1$ & $1$ \\ \hline
$\varphi_3$ & $a_1b_2$ & $a_1b_2$ & $a_2b_1b_2$ & $1$ \\ \hline
$\varphi_4$ & $a_1b_2$ & $a_1b_2$ & $a_1a_2b_1$ & $1$ \\ \hline
$\varphi_5$ & $a_1a_2b_1$ & $a_1a_2b_1$ & $a_2b_1b_2$ & $1$ \\ \hline
$\varphi_6$ & $a_1a_2b_1$ & $a_1a_2b_1$ & $a_1b_2$ & $1$ \\ \hline
\end{tabular}
\end{center}

Unlike for the other normalized endomorphisms, the set $\{\varphi_i\mid i=1,\ldots,6\}$ is \textit{not} closed under $\circ$; it \textit{is} closed modulo normalization though. The following table shows the values of the compositions of the $\varphi_i$ (the cell in the $i$-th row and $j$-th column shows the composition $\varphi_i\circ\varphi_j$). The (endomorphisms mapping into the center represented by) matrices by which one has to shift the $\varphi_i$ to get the composition values all satisfy the property that each of their columns is either the zero vector or the vector $(1,1,0,1)^t$; among those matrices, let $M_1$ denote the one where precisely the last two columns are zero, $M_2$ the one where only the last column is zero, and $M_3:=M_1+M_2$. Then we have:

\begin{center}
\begin{tabular}{|c|c|c|c|c|c|c|}
\hline
$\circ$ & $\varphi_1$ & $\varphi_2$ & $\varphi_3$ & $\varphi_4$ & $\varphi_5$ & $\varphi_6$ \\ \hline
$\varphi_1$ & $\varphi_5+M_1$ & $\varphi_6+M_2$ & $\varphi_2+M_3$ & $\varphi_1+M_3$ & $\varphi_4+M_2$ & $\varphi_3+M_1$ \\ \hline
$\varphi_2$ & $\varphi_3$ & $\varphi_4+M_3$ & $\varphi_1$ & $\varphi_2+M_3$ & $\varphi_6+M_1$ & $\varphi_5+M_1$ \\ \hline
$\varphi_3$ & $\varphi_6$ & $\varphi_5+M_3$ & $\varphi_4$ & $\varphi_3+M_3$ & $\varphi_2+M_1$ & $\varphi_1+M_1$ \\ \hline
$\varphi_4$ & $\varphi_1+M_1$ & $\varphi_2+M_2$ & $\varphi_3+M_3$ & $\varphi_4+M_3$ & $\varphi_5+M_2$ & $\varphi_6+M_1$ \\ \hline
$\varphi_5$ & $\varphi_4+M_1$ & $\varphi_3+M_2$ & $\varphi_6+M_3$ & $\varphi_5+M_3$ & $\varphi_1+M_2$ & $\varphi_2+M_1$ \\ \hline
$\varphi_6$ & $\varphi_2$ & $\varphi_1+M_3$ & $\varphi_5$ & $\varphi_6+M_3$ & $\varphi_3+M_1$ & $\varphi_4+M_1$ \\ \hline
\end{tabular}
\end{center}

Modulo normalization, this is a Cayley table of the symmetric group $S_3$. The $\varphi_i$ correspond to the six different permutations of the normalization classes of the three elements $a_1b_2,a_2a_2b_1$ and $a_2b_1b_2$. Note that there are no endomorphisms $f_1,\ldots,f_6$ of $G_{(1,1)}^{(2)}$ such that the set $\{\varphi_i+f_i\mid i=1,\ldots,6\}$ is closed under $\circ$ (so that we cannot express $\mathrm{End}(G_{(1,1)}^{(2)})$ as a two-sided semidirect product as for the $G_{(1,0)}^{(p)}$), since otherwise, because $\varphi_4+f_4$ is a neutral element for all the $\varphi_i+f_i$, we would have $$\varphi_i+f_i=(\varphi_i+f_i)\circ(\varphi_4+f_4)=\varphi_i\circ\varphi_4+f_i\circ\varphi_4=\varphi_i+g+f_i\circ\varphi_4,$$ where $g$ is the endomorphism represented by the matrix $M_3$ from above. In view of the definition of $\varphi_4$ and denoting by $c_j$ the $j$-th column of the matrix representing $f_i$, this would yield $$c_1+c_4=c_1, c_1+c_4=c_2, c_1+c_2+c_3+(1,1,0,1)^t=c_3,$$ and thus successively $c_4=(0,0,0,0)^t,c_1=c_2$ and $(1,1,0,1)^t=(0,0,0,0)^t$, a contradiction.

Finally, just as for the $G_{(1,0)}^{(p)}$, we can show:

\begin{propposition}\label{lastProp}
The action of $\mathrm{Aut}(G_{(1,1)}^{(2)})$ on $\mathrm{End}(G_{(1,1)}^{(2)})$ by conjugation has precisely $3\cdot 2^{16}-2^{12}+6\cdot 2^8=194048$ orbits, of which $2\cdot 2^{16}=131072$ have length $1$, $2^{16}-2^{12}=61440$ have length $16$ and $6\cdot 2^8=1536$ have length $256$.
\end{propposition}

\begin{proof}
The 17 normalization classes which do not contain any of the $\varphi_i$ yield $2\cdot 2^{16}$ orbits of length $1$ and $2^{16}-2^{12}$ orbits of length $2^4$, just as for the $G_{(1,0)}^{(p)}$. As for the six additional normalization classes, note that since the rank of the image of any $\varphi_i$ under the canonical projection to the abelianization is $2$, we have more variation here, yielding only orbits of length $2^8$.
\end{proof}

\begin{remmark}\label{veryLastRem}
Our results show that for JK-groups, the number of conjugacy classes of the endomorphism monoid is roughly quadratic in their order $p^8$. In general, however, the order of the automorphism group of a finite aut-abelian $p$-group $G$ (and hence a fortiori the number of conjugacy classes of the endomorphism monoid) cannot be bounded by a polynomial in the order of $G$. This follows directly from the already mentioned results of Morigi \cite{Mor94a}, by observing that $\frac{(n^2+n+1)(2n+2)}{n^2+3n+3}$ is not $O(1)$.
\end{remmark}

\section{Acknowledgements}

The author would like to thank Peter Hellekalek for his helpful comments and Benjamin Steinberg for pointing out the paper \cite{RT89a} to him in a discussion on the Internet site mathoverflow (\url{http://mathoverflow.net/questions/180376/}).

\end{document}